\documentclass[psamsfonts,preprint]{amsart}

\usepackage{amssymb,amsfonts}
\usepackage[all,arc]{xy}
\usepackage{enumerate, bbm}
\usepackage{mathrsfs}
\usepackage{mathtools}
\usepackage{tabularx,ragged2e,booktabs,caption}
\usepackage{breqn}
\usepackage{hyperref}
\usepackage[usenames, dvipsnames]{color}
\usepackage{graphicx}
\graphicspath{ {TexImages/} }

\newtheorem{thm}{Theorem}[section]
\newtheorem{cor}[thm]{Corollary}
\newtheorem{prop}[thm]{Proposition}
\newtheorem{lem}[thm]{Lemma}
\theoremstyle{definition}

\newtheorem{quest}[thm]{Question}
\newtheorem{conj}[thm]{Conjecture}\theoremstyle{remark}

\newcommand{\R}{\mathbb{R}}

\newcommand{\F}{\mathbb{F}}

\newcommand{\sub}{\subseteq}

\newcommand{\NL}{\mathcal{L}}
\newcommand{\pf}{\tilde{v}}
\newcommand{\Eig}[2]{E_{#2}(#1)}

\newcommand{\q}[1]{[#1]_q}
\newcommand{\qchoose}[2]{{#1\brack #2}_q}

\hypersetup{
	colorlinks,
	citecolor=red,
	filecolor=red,
	linkcolor=red,
	urlcolor=red,
	linktocpage
}

\title{Polynomial Relations between Matrices of Graphs}

\author{Sam Spiro}
\address{University of California, San Diego}

\begin{document}
	\begin{abstract}
		We derive a correspondence between the eigenvalues of the adjacency matrix $A$ and the signless Laplacian matrix $Q$ of a graph $G$ when $G$ is $(d_1,d_2)$-biregular by using the relation $A^2=(Q-d_1I)(Q-d_2I)$.  This motivates asking when it is possible to have $X^r=f(Y)$ for $f$ a polynomial, $r>0$, and $X,\ Y$ matrices associated to a graph $G$.  It turns out that, essentially, this can only happen if $G$ is either regular or biregular.
	\end{abstract}
	\maketitle
	
	\section{Introduction}
	For $G$ a simple graph, we let $A$ denote the adjacency matrix of $G$ and $D$ the diagonal matrix of vertex degrees of $G$.  We define the signless Laplacian matrix $Q$ of $G$ by $Q=D+A$, the Laplacian matrix $L$ of $G$ by $L=D-A$, and when $G$ has no isolated vertices, we define the normalized Laplacian matrix $\NL$ of $G$ by $D^{-1/2}LD^{-1/2}$.  For more detailed information about these matrices see, for example, \cite{chungButler}.
	
	A graph $G$ is said to be $(d_1,d_2)$-biregular (sometimes called semiregular) if $V_1\cup V_2$ is a partition of the vertices of $G$ such that no two vertices of $V_i$ are adjacent to one another (so $G$ is bipartite), and for all $v\in V_i,\ d_v=d_i$.  A graph is said to be biregular if it is $(d_1,d_2)$-biregular for some $d_1,d_2$.
	
	When $G$ is biregular it is possible to directly relate the eigenvalues of $A$ and $Q$ through the following formula of \cite{signless}.
	
	\begin{thm}\label{T-signlessPolynomial}
		If $G$ is $(d_1,d_2)$-biregular with $|V_i|=n_i,\ n_1\ge n_2$, and $\lambda_1,\ldots,\lambda_{n_2}$ are the $n_2$ largest eigenvalues of $A$ in decreasing order,  then 
		\[
		Q_G(x)=x(x-d_1-d_2)(x-d_1)^{n_1-n_2}\prod_{i=2}^{n_2}((x-d_1)(x-d_2)-\lambda_i^2),
		\]
		where $Q_G(x)$ denotes the characteristic polynomial of $Q$.
	\end{thm}
	
	Theorem~\ref{T-signlessPolynomial} was proved in \cite{signless} by using arguments involving the line graph of $G$.  In this paper we present an independent derivation of this theorem using the relation
	\[
	A^2=(Q-d_1I)(Q-d_2I).
	\]
	
	Given this derivation, it is natural to ask what other graphs $G$ satisfy $A^r=f(Q)$ for some polynomial $f$ and positive integer $r$. More generally, one can ask what $G$ satisfy $X^r=f(Y)$ when $X$ and $Y$ are matrices associated to the graph  $G$.  Of the cases we consider, the only graphs found to have this property are graphs that are either regular or biregular.  We summarize our results in the following theorem, where we note that the first part of the theorem is clear from the definitions of $Q,\ L$ and $\NL$.
	
	\begin{thm}
		Let $G$ be a connected graph, $r$ a positive integer and $f$ a polynomial.
		\begin{itemize}
			\item If $G$ is regular, then $X^r=f(Y)$ can occur if $X,\ Y$ are any  of $A,\ Q,\ L$ or $\NL$.  Moreover, all of these matrices can be related to one another by a linear equation.
			\item If $G$ is $(d_1,d_2)$-biregular, then $X^r=f(Y)$ can occur if $X=A$ and $Y=Q,\ L$, or $\NL$, or when $X=\NL$ and $Y=A$. Specifically, we have
			\begin{align*}
			(Q-d_1I)(Q-d_2I)&=(L-d_1I)(L-d_2I)=A^2\\ 
			\mathcal{L}&=I-\frac{1}{\sqrt{d_1d_2}}A.
			\end{align*}
			Moreover, if $X,\ Y$ are any other pair from $A,\ Q,\ L, \NL$ then no such  relation exists.
			\item If $G$ is not regular or biregular, then $X^r=f(Y)$ can not hold if $X,\ Y$ are any distinct matrices of $A,\ Q,\ L,\ \NL$, except possibly for the case $A^r=f(\NL)$.
		\end{itemize}
	\end{thm}
	
	We establish the following conventions.  Whenever $X^r=f(Y)$ is written it is assumed that $f$ is a polynomial and $r$ is a positive integer.  We assume throughout this paper that $G$ is a connected graph, though we emphasize this point in the statement of our theorems. $\mathbf{1}$ will denote the vector of all 1's.  $m_i(u,v)$ will denote the number of walks of length $i$ between the vertices $u$ and $v$ in the graph $G$. We note that $(A^i)_{uv}=m_i(u,v)$ (see Theorem 1.1 of \cite{stanleyBook}, for example).  For a matrix $M$ we let $\Eig{M}{\lambda}$ denote the eigenspace of $M$ with corresponding eigenvalue $\lambda$.  $V(G)$ and $E(G)$ will denote the set of vertices and the set of edges of the graph  $G$ respectively.
	
	The structure of the paper is as follows.  In Section~\ref{S-rel} we derive Theorem~\ref{T-signlessPolynomial} and in Section~\ref{S-tree} we apply this theorem to count the number of spanning trees of biregular graphs.  In Sections \ref{S-Q} and \ref{S-NL} we establish necessary conditions for $G$ to satisfy $X^r=f(Y)$ with $X,\ Y$ equal to $A,\ Q,\ L$, and $\NL$.  Lastly, in Section~\ref{S-gen} we briefly explore the more general question of establishing relations of the form $f(X)=g(Y)$ where $f$ and $g$ are both polynomials and $X$ and $Y$ are matrices associated to a graph $G$.

	\section{Relating Eigenvalues of $A$ and $Q$}\label{S-rel}
	
	\begin{prop}\label{P-keyrel}
		If $G$ is a $(d_1,d_2)$-biregular graph, then
		\[
		A^2=(Q-d_1I)(Q-d_2I).
		\]
	\end{prop}
	\begin{proof}
		Let $Q'=(Q-d_1I)(Q-d_2I)$.  By definition, $Q'_{uv}$ is equal to the dot product of the $u$th row of $Q-d_1 I$ with the $v$th column of $Q-d_2I$.  From these definitions we have that \begin{align*}
		Q'_{uv}=(d_u-d_1)m_1(u,v)+(d_v-d_2)m_1(v,u)+&\sum_{w\ne u,v}m_1(u,w)m_1(w,v)\\ =(d_u+d_v-d_1-d_2)m_1(u,v)+&\sum_{w\ne u,v}m_1(u,w)m_1(w,v).
		\end{align*}  If $u,v\in V_i,$ then $m_1(u,v)=0$ and we are left with $\sum_{w\ne u,v}m_1(u,w)m_1(w,v)=m_2(u,v)$.  If, say, $u\in V_1,v\in V_2$, then $m_1(u,w)m_1(w,v)=0$ for all $w\ne u,v$, and further, $d_u+d_v-d_1-d_2=0$.  Thus in this case $Q'_{uv}=0=m_2(u,v)$ (since the graph is bipartite and $u,v$ belong to different partition classes).  We conclude that for all $u,v$ that $Q'_{uv}=m_2(u,v)=(A^2)_{uv}$, completing the proof.
	\end{proof}
	We note that through an analogous computation one can show that if $G$ is biregular, then \[A^2=(L-d_1I)(L-d_2I).\]
	We also note that every result in this section remains valid, except for straightforward changes of some signs, if one replaces $Q$ with $L$.

	\begin{cor}\label{C-QtoA}
		Let $G$ be a $(d_1,d_2)$-biregular graph.  If $q_1,\ldots,q_n$ are the eigenvalues of $Q$, then $(q_1-d_1)(q_1-d_2),\ldots,(q_n-d_1)(q_n-d_2)$ are the eigenvalues of $A^2$.  
		
		Moreover, if $\lambda$ is an eigenvalue of $A^2$ and $q_1,q_2$ are the solutions to the equation $\lambda=(x-d_1)(x-d_2)$, then $\Eig{Q}{q_1}\oplus\Eig{Q}{q_2}=\Eig{A^2}{\lambda}$.
	\end{cor}
	\begin{proof}
		Let $V'=\{v_1,\ldots,v_n\}$ be a basis of eigenvectors of $Q$ with $Qv_i=q_i v_i$, which exists because $Q$ is symmetric and hence diagonalizable.  Then \[A^2v_i=(Q-d_1I)(Q-d_2I)v_i=(q_i-d_1)(q_i-d_2)v_i,\] so $v_i$ will be an eigenvector of $A^2$ with the desired eigenvalue.  From this it is also clear that a basis for $\Eig{Q}{q_1}\oplus \Eig{Q}{q_2}$ is also a basis for $\Eig{A}{\lambda}$ when $q_1,q_2$ are the solutions to $\lambda=(x-d_1)(x-d_2)$.
	\end{proof}
	From Corollary~\ref{C-QtoA} it is possible to translate from the eigenvalues of $Q$ to the eigenvalues of $A$ when $G$ is biregular.  Namely, because $G$ is bipartite, $A$'s spectrum will be symmetric about 0 (see Proposition 3.4.1 of \cite{brouwer}), so knowing the eigenvalues (with multiplicity) of $A^2$ is equivalent to knowing the eigenvalues (with multiplicity) of $A$.  
	
	What is less obvious is that the converse of the above statement is true.  That is, given the eigenvalues of $A$ when $G$ is biregular, one can compute the eigenvalues of $Q$.  Certainly we know that if $\lambda^2$ is an eigenvalue of $A^2$ then $Q$ must have an eigenvalue $q$ satisfying $(q-d_1)(q-d_2)=\lambda^2$, but if $d_1\ne d_2$ then it is not clear which root of this equation correctly corresponds to the eigenvalue in $Q$ (if $d_1=d_2$ then $G$ is regular and there is only one root to choose).  To figure out the multiplicities of the eigenvalues of $Q$ we will need the following lemma.
	\begin{lem}\label{L-nonint}
		Let $G$ be $(d_1,d_2)$-biregular with $d_1\ne d_2$ and let $v$ be an eigenvector of $A$ with eigenvalue $\lambda\ne 0$.  Then $v$ is not an eigenvector of $Q$.
	\end{lem}
	\begin{proof}
		Assume that $Qv=\mu v$.  Then $Dv=(Q-A)v=(\mu-\lambda)v$, so $v$ is also an eigenvalue of $D$.  This implies that $\mu-\lambda=d_i$ for $i=1$ or 2, and hence the set $V'=\{u:v_u\ne 0\}$ lies entirely in the corresponding $V_i$.  But if $u\in V'$ then $(Av)_u=0$, as all of the neighbors of $u$ belong to the other partition class and hence are given 0 weight in $v$. Since $v_u\ne 0$ and $\lambda v_u=(Av)u=0$, we must have $\lambda=0$, contradicting the assumption that this is not the case.
	\end{proof}
	
	\begin{lem}\label{L-split}
		For $G$ a $(d_1,d_2)$-biregular graph, let $\lambda^2\ne 0$ be an eigenvalue of $A^2$ with $\dim \Eig{A^2}{\lambda^2}=m$ and let $q_1,q_2$ be the two roots of the equation $\lambda^2=(x-d_1)(x-d_2)$.  Then $\dim \Eig{Q}{q_1}=\dim \Eig{Q}{q_2}=m/2$.
	\end{lem}
	\begin{proof}
		Note first that since $\lambda^2\ne 0$ and $G$ is bipartite, $m$ is even and $m/2$ is an integer.  Moreover, $\dim \Eig{A}{\lambda}=\dim \Eig{A}{-\lambda}=m/2$ and $\Eig{A}{\lambda}\oplus\Eig{A}{-\lambda}= \Eig{A^2}{\lambda^2}$.  By Corollary~\ref{C-QtoA} we have $\Eig{Q}{q_1}\oplus \Eig{Q}{q_2}= \Eig{A^2}{\lambda^2}$.  If $\dim \Eig{Q}{q_i}>m/2$, then we must have $\Eig{Q}{q_i}\cap \Eig{A}{\lambda}\ne \{0\}$ by a dimensionality argument, but this can't happen by Lemma~\ref{L-nonint} and from the assumption that $\lambda\ne 0$.  We conclude that $\dim \Eig{Q}{q_i}\le m/2$, and since $\dim \Eig{Q}{q_1}+\dim \Eig{Q}{q_2}=m$, we must have $\dim \Eig{Q}{q_1}=\dim \Eig{Q}{q_2}=m/2$.
	\end{proof}
	
	\begin{lem}\label{L-zeroes}
		If $G$ is a biregular graph with $|V_1|=n_1,\ |V_2|=n_2,\ n_1\ge n_2$ and $\dim \Eig{A}{0}=m$, then $\dim \Eig{Q}{d_1}=n_1-n_2+k,\ \dim \Eig{Q}{d_2}=k$ where $k=\frac{m-(n_1-n_2)}{2}$.
	\end{lem}
	\begin{proof}
		Let $Z_1$ denote the set of null-vectors of $A$ whose non-zero coordinates lie entirely in $V_1$, and similarly define $Z_2$.  It is not difficult to see that $Z_1\oplus Z_2=\Eig{A}{0}$.  Moreover, if $v\in Z_i$ then \[d_iv=Dv=Qv-Av=Qv,\] so $Z_i\sub \Eig{Q}{d_i}$.  Since $d_1,\ d_2$ are the unique roots of $0=(x-d_1)(x-d_2)$, it follows from Corollary~\ref{C-QtoA} that \[\Eig{Q}{d_1}\oplus \Eig{Q}{d_2}=\Eig{A}{0}=Z_1\oplus Z_2,\] and this implies that $Z_i=\Eig{Q}{d_i}$.  Thus it will be sufficient to prove that $\dim Z_1-\dim Z_2=n_1-n_2$.
		
		Let $M$ be the $n_1\times n_2$ sub-matrix of $A$ whose rows are indexed by $V_1$ and whose columns are indexed by $V_2$.  Let $r$ be the rank of this matrix.  Then the null-space of $M$ has dimension $n_2-r$.  Moreover if $v\in Z_2$, one can construct a vector $v'$ in the null space of $M$ by setting $v'_u=v_u$.  It isn't difficult to see that the correspondence between $v$ and $v'$ is a bijection between vectors of $Z_2$ and null-vectors of $M$, and moreover this mapping implies that $\dim Z_2=n_2-r$. The same argument on $M^T$ shows that $\dim Z_1=n_1-r$, and hence that $\dim Z_1-\dim Z_2=n_1-n_2$, proving the statement.
	\end{proof}
	
	\begin{proof}[Proof of Theorem~\ref{T-signlessPolynomial}]
		The characteristic polynomial of $Q$ is the monic polynomial whose roots are the eigenvalues of $Q$ with corresponding multiplicity.  For each positive eigenvalue $\lambda$ of $A$, the two roots of $(x-d_1)(x-d_2)-\lambda^2$ will be eigenvalues of  $Q$ by Lemma~\ref{L-split}.  Note that this will account for all of the eigenvalues of $Q$ except for the eigenvalues $d_1$ and $d_2$.  Also note that all of the positive eigenvalues of $A$ are included in the $n_2$ largest eigenvalues of $A$ because $G$ is bipartite.  
		
		If $A$ has $n_2-k$ positive eigenvalues, then it must have $n_1-n_2+2k$ eigenvalues equal to 0, meaning $Q$ has $d_1$ as an eigenvalue with multiplicity $n_1-n_2+k$ and $d_2$ with multiplicity $k$ by Lemma~\ref{L-zeroes}.  Thus if $\lambda_1,\ldots,\lambda_{n_2}$ are the $n_2$ largest eigenvalues of $A$, the eigenvalues of $Q$ agree with the roots of
		\begin{align*}
		(x-d_1)^{n_1-n_2}&\left(\prod_{i=1}^k (x-d_1)(x-d_2)\right)\left(\prod_{i=1}^{n_2-k}((x-d_1)(x-d_2)-\lambda_i^2)\right)\\ 
		&=(x-d_1)^{n_1-n_2}\prod_{i=1}^{n_2}((x-d_1)(x-d_2)-\lambda_i^2),
		\end{align*}
		so this must equal $Q_G(x)$.
		
		We lastly note the following fact stated in \cite{signless}: if $G$ is a connected $(d_1,d_2)$-biregular graph, then $\sqrt{d_1d_2}$ will be the largest eigenvalue of $A$, and the two roots of $d_1d_2=(x-d_1)(x-d_2)$ are $x=0$ and $x=d_1+d_2$.  Thus to get the exact form as written in Theorem~\ref{T-signlessPolynomial} we simply pull out the factor $(x-d_1)(x-d_2)-\lambda_1^2=x(x-d_1-d_2)$ from the product.
	\end{proof}
	
	\section{Spanning Trees}\label{S-tree}
	We provide an application of Theorem~\ref{T-signlessPolynomial}, namely that of counting spanning trees of biregular graphs. Our main tool will be the Matrix-Tree theorem, a proof of which can be found in \cite{stanleyBook}. 
	
	\begin{thm}[Matrix-Tree theorem]
		Let $\mu_1=0,\mu_2,\ldots,\mu_n$ denote the eigenvalues of the Laplacian matrix $L$ of $G$. The number of spanning trees of $G$ is equal to \[\frac{\mu_2\cdot \mu_3\cdots \mu_n}{|V(G)|}.\]
	\end{thm}
	
	If $G$ is bipartite then $Q$ and $L$ will have the same spectrum (see Proposition~1.3.10 of \cite{brouwer}).  Thus if we can compute the eigenvalues of $A$ when $G$ is biregular, we can use our previous results to obtain the eigenvalues of $Q$, and hence of $L$, in order to compute the number of spanning trees of $G$ by the Matrix-Tree theorem.
	
	\begin{thm}\label{T-application}
		If $G$ is a $(d_1,d_2)$-biregular graph with $|V_i|=n_i,\ n_1\ge n_2,$ and $\lambda_1,\ldots,\lambda_{n_2}$ are the largest eigenvalues of $A$, then the number of spanning trees of $G$ will be
		\[
		\frac{(d_1+d_2)d_1^{n_1-n_2}\prod_{i=2}^{n_2}(d_1d_2-\lambda_i^2)}{n_1+n_2}.
		\]
	\end{thm}
	\begin{proof}
		By the Matrix-Tree theorem, the number of spanning trees of $G$ will be equal to the product of the $n-1$ largest eigenvalues of $L$ divided by $n_1+n_2$.  Since $G$ is bipartite, this is equivalent to taking the product of the eigenvalues of $Q$ after ignoring a 0 eigenvalue and dividing by $n_1+n_2$, and this will simply be $\frac{Q_G(x)}{(n_1+n_2)x}$ evaluated at $x=0$. By using this and Theorem~\ref{T-signlessPolynomial}, one arrives at the desired result.
	\end{proof}
	
	Let $C_{n}$ denote the $n$-cube, i.e. the graph whose vertices are $n$-length bit strings and two strings are adjacent if their hamming distance is 1.  Define $C_{n,k}$ to be the subgraph of $C_n$ induced by all vertices of $C_n$ that have either $k-1$ or $k$ 1's. 
	
	\begin{thm}
		The number of spanning trees of $C_{n,k}$ when $k\le n/2$ is
		
		\[
		\frac{(n+1)k^{{n\choose k}-{n\choose k-1}}\prod_{i=1}^{k-1}((k-i)(i+n-k+1))^{{n\choose k-i}-{n\choose k-i-1}}}{{n\choose k}+{n\choose k-1}}. 
		\]
	\end{thm}
	\begin{proof}
		From the definition of $C_{n,k}$ it is clear that this graph is $(k,n-k+1)$-biregular with $|V_1|={n\choose k},\ |V_2|={n\choose k-1}$, and we have $|V_1|\ge |V_2|$ since $k\le n/2$. It was proven in Theorem~2.12 of \cite{stanleyPaper} that the squares of the $|V_2|$ largest eigenvalues of the adjacency matrix of $C_{n,k}$ are $i(n-2k+i+1)$ for $1\le i\le k$, each having multiplicity ${n\choose k-i}-{n\choose k-i-1}$.  The result follows after  applying Theorem~\ref{T-application} and observing that \[k(n-k+1)-i(n-2k+i+1)=(k-i)(i+n-k+1).\]
	\end{proof}
	
	More generally, let $C_n(q)$ be the lattice of subspaces of an $n$-dimensional vector space over the finite field $\F_q$.  Let $C_{n,k}(q)$ denote the graph whose vertices are the elements of $C_n(q)$ of dimensions $k$ and $k-1$ with two vertices being adjacent if one is a subspace of the other (thus this is the Hasse graph of $C_n(q)$ induced by the elements of rank $k$ and $k-1$).  Let
	\begin{align*}
	\q{n}&=1+q+\cdots+q^{n-1}=\frac{q^n-1}{q-1},\\  \qchoose{n}{k}&=\frac{(q^n-1)(q^{n-1}-1)\cdots(q^{n-k+1}-1)}{(q^k-1)(q^{k-1}-1)\cdots(q-1)}.
	\end{align*}
	
	\begin{thm}
		The number of spanning trees of $C_{n,k}(q)$ when $k\le n/2$ is
		
		\[
		\frac{(\q{k}+\q{n-k+1})(\q{k})^{\qchoose{n}{k}-\qchoose{n}{k-1}}\prod_{i=1}^{k-1}(\q{k}\q{n-k+1}-\gamma_i)^{\qchoose{n}{k-i}-\qchoose{n}{k-i-1}}}{\qchoose{n}{k}+\qchoose{n}{k-1}}, 
		\]
		where $\gamma_i=\q{i}(q^{k-i}\q{n-2k}+q^{n-k-i}\q{i+1})$.
	\end{thm}
	\begin{proof}
		$C_{n,k}(q)$ is $(\q{k},\q{n-k+1})$-biregular with $|V_1|=\qchoose{n}{k},\ |V_2|=\qchoose{n}{k-1}$, and  $|V_1|\ge |V_2|$.  It was proven in Theorem~2.12 of \cite{stanleyPaper} that the squares of the $|V_2|$ largest eigenvalues of the adjacency matrix of $C_{n,k}(q)$ are, for $1\le i\le k$, \begin{align*}r_{k-1}+r_{k-2}+\cdots+r_{k-i},\textrm{ with}\\  r_i=\q{n-i}-\q{i}=q^i\q{n-2i},\end{align*} each  with multiplicity $\qchoose{n}{k-i}-\qchoose{n}{k-i-1}$.  We wish to put these expressions into a closed form.
		
		We have 
		\[\sum_{s=1}^i r_{k-s}=\sum_{s=1}^i q^{k-s}\q{n+2s-2k}=\sum_{s=1}^{i}q^{k-s}(\q{n-2k}+q^{n-2k}\q{2s}),\] 
		so it will be sufficient to find closed forms for the sums $\sum_{s=1}^{i}q^{k-s}\q{n-2k}$ and $q^{n-k}\sum_{s=1}^iq^{-s}\q{2s}$.  The first sum can be written as \begin{align*}
		\q{n-2k}\sum_{s=1}^iq^{k-s}&=\q{n-2k}\sum_{s=1}^i q^{k-i+s-1}=q^{k-i}\q{n-2k}\sum_{s=1}^iq^{s-1}\\ &=q^{k-i}\q{n-2k}\q{i}.
		\end{align*}
		
		For the second sum,
		\begin{align*}
		&q^{n-k}\sum_{s=1}^iq^{-s}\q{2s}=q^{n-k}\sum_{s=1}^i q^{-s}\frac{q^{2s}-1}{q-1}=\frac{q^{n-k}}{q-1}\sum_{s=1}^i q^s-q^{-s} \\&=\frac{q^{n-k}}{q-1}\left(\frac{q(q^i-1)}{q-1}-\frac{q^{-i}(q^i-1)}{q-1}\right) =\frac{q^{n-k-i}(q^{i+1}-1)(q^{i}-1)}{(q-1)^2}\\ &=q^{n-k-i}\q{i+1}\q{i}.
		\end{align*}
		Thus in total the squares of the eigenvalues are of the form \[q^{k-i}\q{n-2k}\q{i}+q^{n-k-i}\q{i+1}\q{i}=\q{i}(q^{k-i}\q{n-2k}+q^{n-k-i}\q{i+1})=\gamma_i,\]  and plugging this into Theorem~\ref{T-application} gives the desired result.
		
	\end{proof}

	\section{Relations Involving $A,\ Q,$ and $L$}\label{S-Q}
	When $G$ is biregular, we proved that there exists a relation of the form $A^r=f(Q)$ that allows us to  translate between eigenvalues of $A$ and eigenvalues of $Q$, and if $G$ is $d$-regular, the relation $A=Q-dI$ gives an analogous result.  One might hope that there exists some notion of ``tripartite'' graphs for which a similar result holds.  However, it turns out that the only graphs that can satisfy $A^r=f(Q)$ are the regular and biregular graphs.
	
	The general idea in proving that $X^r=f(Y)$ implies that the underlying graph $G$ has a certain property $P$ is as follows.  We first show that if $X$ and $Y$ share a certain eigenvector $v$, then $G$ must have property $P$.  We then use the following three lemmas to show that if $X^r=f(Y)$, then $X$ and $Y$ both have $v$ as an eigenvector.
	
	We note that $Q,\ L,$ and $\NL$ have nonnegative spectrum (see \cite{chungButler}, for example), and that $A$'s spectrum is real, so $A^2$ has nonnegative spectrum.
	
	\begin{lem}\label{L-same}
		Let $X$ be a diagonalizable matrix with nonnegative spectrum (such as $A^2,\ L,\ Q,$ or $\NL$).  Assume that $X^r=f(Y)$ for some matrix $Y$.  If $v$ is an eigenvector of $Y$, then $v$ is an eigenvector of $X$.
	\end{lem}
	\begin{proof}
		If $V'=\{v_1,\ldots,v_n\}$ is a basis of eigenvectors of $X$ with $Xv_i=\mu_iv_i$, then $X^r v_i=\mu_i^r v_i$ for all $i$, so $V'$ will also be a basis of eigenvectors of $X^r$.  It follows that $\Eig{X^r}{\mu}=\bigoplus_{\mu_i^r=\mu} \Eig{X}{\mu_i}$ for all eigenvalues $\mu$ of $X^r$.  As $\mu_i\ge 0$ for all $i$ by assumption, we must have $\Eig{X^r}{\mu}=\Eig{X}{\mu^{1/r}}$ for all eigenvalues of $X^r$.  Thus any eigenvector of $X^r$ is also an eigenvector of $X$.  But if $v$ is an eigenvector of $Y$ with eignevalue $\lambda$, then $X^rv=f(Y)v=f(\lambda)v$.  Thus $v$ is an eigenvector of $X^r$, and hence of $X$.
	\end{proof}
	
	\begin{lem}\label{L-onedim}
		Let $X,\ Y$ be diagonalizable matrices such that $X^r=f(Y)$, and assume that there exists a $\mu$ such that $\Eig{X}{\mu}=\Eig{X^r}{\mu^r}$ with $\dim \Eig{X}{\mu}=1$.  If $v\in \Eig{X}{\mu}$, then $v$ is an eigenvector of $Y$.
	\end{lem}
	\begin{proof}
		Let $V'=\{v_1,\ldots,v_n\}$ be a basis of eigenvectors of $Y$.  This will also be a basis of eigenvectors of $X^r$, so there exists a vector $v_i\in V'$ such that $v_i\in \Eig{X^r}{\mu^r}= \Eig{X}{\mu}$.  Since $\dim \Eig{X}{\mu}=1$, we conclude that $v_i$ is a scaler multiple of $v$, and hence $v$ is also an eigenvector of $Y$.
	\end{proof}
	
	One can strengthen the previous lemma if both matrices have nonnegative spectrum.
	\begin{lem}\label{L-bothNonneg}
		Let $X,\ Y$ be diagonalizable matrices with nonnegative spectrum and assume that there exists a $\mu$ such that $\dim \Eig{X}{\mu}=1$ with $v\in \Eig{X}{\mu}$. If either $X^r=f(Y)$ or $Y^r=f(X)$, then $v$ will be an eigenvector of $Y$.
	\end{lem}
	\begin{proof}
		The case $X^r=f(Y)$ follows from Lemma~\ref{L-onedim} after one notes that $\Eig{X^r}{\mu^r}=\Eig{X}{\mu}$ because the spectrum of $X$ is nonnegative.  The case $Y^r=f(X)$ follows from Lemma~\ref{L-same} because $Y$ has nonnegative spectrum.
	\end{proof}

	We recall the Perron-Frobenius theorem.
	
	\begin{thm}[Perron-Frobenius]\label{T-pf}
		Let $M$ be an irreducible matrix with nonnegative entries.  If $\Lambda$ is the largest eigenvalue of $M$, then it has multiplicity one and there exists an eigenvector $\pf$ with $M\pf=\Lambda \pf$ such that every entry of $\pf$ is positive.
	\end{thm}
	If $G$ is connected (which we always assume to be the case), then Theorem~\ref{T-pf} applies to $A$ and $Q$.  
	
	It turns out that the key lemmas needed to prove necessary conditions for $A^r=f(Q)$ are the same lemmas needed to prove necessary conditions for $X^r=f(Y)$ when $X$ and $Y$ are \textit{any} two matrices of $A,\ Q,$ and $L$, so we shall generalize our notation to deal with all of these cases at the same time.
	
	To this end, we will say that $(N,P)$ is a \textit{Laplacian pair} if $N$ is a nonnegative irreducible diagonalizable matrix, $P$ is a diagonalizable matrix with nonnegative spectrum, and $N+aP=bD$ for some $a,b\in \R\setminus\{0\}$.  We note that $(A,Q),\ (A,L)$ and $(Q,L)$ are all Laplacian pairs, since we have $A-Q=-D,\ A+L=D,\ Q+L=2D$, and the other conditions are all clearly satisfied.
	
	Given a Laplacian pair $(N,P)$, we will let $\Lambda$ refer to the largest eigenvalue of $N$ and $\pf$ will refer to its corresponding positive eigenvector as is guaranteed by the Perron-Frobenius theorem.
	
	\begin{lem}\label{L-HW}
		Let $(N,P)$ be a Laplacian pair.  If $\pf$ is also an eigenvector of $P$, then $G$ is regular.
	\end{lem}
	\begin{proof}
		Assume that $P\pf=\mu \pf $ for some $\mu$.  Then $bD\pf=(aP+N)\pf=(a\mu+ \Lambda)\pf$, so $\pf$ is also an eigenvector of $D$.  But the only way for $\pf$ to be an eigenvector of $D$ is if each of its non-zero coordinates have the same degree in $G$, and since every coordinate of $\pf$ is non-zero, this implies that $G$ is regular. 
	\end{proof}
	
	\begin{thm}\label{T-pair}
		If $(N,P)$ is a Laplacian pair and $P^r=f(N)$, then $G$ is regular.
	\end{thm}
	\begin{proof}
		If $P^r=f(N)$, then $\pf$ will be an eigenvector of $P$ by Lemma~\ref{L-onedim} (as $\pf\in \Eig{N}{\Lambda},\ \dim\Eig{N}{\Lambda}=1$, and $P$ has nonnegative spectrum by definition of $(N,P)$ being a Laplacian pair).  $G$ being regular then follows from Lemma~\ref{L-HW}.
	\end{proof}
	\begin{cor}
		If $G$ is connected and $Q^r=f(A),\ L^r=f(A)$, or $L^r=f(Q)$, then $G$ is regular.
	\end{cor}
	\begin{proof}
		$(A,Q),\ (A,L)$, and $(Q,L)$ are all Laplacian pairs, so this immediately follows from Theorem~\ref{T-pair}.
	\end{proof}
	
	\begin{thm}\label{T-LQ}
		If $G$ is connected and $Q^r=f(L)$, then $G$ is regular.
	\end{thm}
	\begin{proof}
		Let $\pf$ be the positive eigenvector of $Q$ guaranteed by the Perron-Frobenius theorem. If we have $Q^r=f(L)$, then we conclude that $\pf$ is an eigenvector of $L$ by Lemma~\ref{L-bothNonneg}.  But $\pf$ being an eigenvector of both $L$ and $Q$ implies  that  $G$ is regular by Lemma~\ref{L-HW}.
	\end{proof}

	We now focus on Laplacian pairs with $N=A$.
	\begin{lem}\label{L-odd bip}
		If $(A,P)$ is a Laplacian pair and $A^r=f(P)$ with either $r$ odd or $G$ not bipartite, then $G$ is regular.
	\end{lem}
	\begin{proof}
		Since $A$ has real spectrum it will always be the case that $\Eig{A^r}{\mu}=\Eig{A}{\mu^{1/r}}$ if $r$ is odd, and $\Eig{A^r}{\mu}=\Eig{A}{\mu^{1/r}}\oplus\Eig{A}{-\mu^{1/r}}$ if $r$ is even.  If $r$ is odd, then in particular we have $\Eig{A}{\Lambda}=\Eig{A^r}{\Lambda^r}$.  If $G$ is not bipartite then $-\Lambda$ is not an eigenvalue of $A$ (see Proposition~3.4.1 of \cite{brouwer}), and hence for all $r$ we have $\Eig{A}{\Lambda}=\Eig{A}{\Lambda}\oplus \Eig{A}{-\Lambda}=\Eig{A^r}{\Lambda^r}$.  As $\dim \Eig{A}{\Lambda}=1$ with $\pf\in \Eig{A}{\Lambda}$, we conclude in either case that $\pf$ is an eigenvector of $P$ by Lemma~\ref{L-onedim}, so $G$ must be regular by Lemma~\ref{L-HW}.
	\end{proof}
	
	For a bipartite graph $G$ with vertex partition $V_1\cup V_2$, let $\mathbf{1}'$ be defined by $\mathbf{1}'_v=1$ if $v\in V_1$ and $\mathbf{1}'_v=-1$ if $v\in V_2$.
	
	\begin{lem}\label{L-1}
		If $G$ is a bipartite graph with vertex partition $V_1\cup V_2$ and either $\mathbf{1}'$ or $\mathbf{1}$ is an eigenvector of $A^2$, then $G$ is biregular.
	\end{lem}
	\begin{proof}
		\[(A^2\mathbf{1}')_v=\sum_{u\in V(G)}(\mathbf{1}'_u) m_2(u,v)=(\mathbf{1}'_v)\sum_{u\in V(G)} m_2(u,v),\] as every vertex that $v$ can reach in two steps belongs to the same partition class as $v$.   Thus $\mathbf{1}'$ will be an eigenvector of $A^2$ iff $\sum_{u\in V(G)} m_2(u,v)$ is equal to the same value for all $v$, and it is clear that this is also an equivalent condition for $\mathbf{1}$ being an eigenvector of $A^2$.  We note that \[\sum_{u\in V(G)} m_2(u,v)=\sum_{uv\in E(G)}d_u,\] as every walk of length two starting from $v$ is characterized by walking along an edge to some $u$ and then taking one of the $d_u$ edges connected to $u$.  Thus $\mathbf{1}'$ or $\mathbf{1}$ is an eigenvector of $A^2$ iff $\sum_{uv\in E(G)}d_u$ is the same value for all $v$.
		
		Assume that there exists a $\lambda$ such that $\lambda= \sum_{uv\in E} d_u$ for all $v$.  Let $v$ be a vertex with minimum degree $d$, and let $v'$ be a vertex with maximum degree $D$.  Then
		\[
		\lambda=\sum_{uv\in E} d_u\le d\cdot D\le \sum_{uv'\in E}d_u=\lambda,
		\]
		where the first inequality follows from the fact that each of the $d$ terms in the sum can have value at most $D$, and the second from the fact that each of the $D$  terms in the sum have value at least $d$.  Since both sides of the inequality are equal, both inequalities must in fact be equalities.  We conclude that if a vertex in $G$ has degree $d$ then all of its neighbors have degree $D$, and conversely if a vertex in $G$ has degree $D$ then all of its neighbors will have degree $d$.  Since $G$ is assumed to be connected, it follows that all vertices must have degree $d$  or $D$.  Moreover, all the vertices of $V_1$ have the same degree, and similarly all the vertices of $V_2$ have the same degree.  Thus $G$ is biregular.
	\end{proof}
	\begin{thm}\label{T-two}
		If $G$ is connected and $A^r=f(Q)$ or $A^r=f(L)$, then $G$ is regular or biregular.
	\end{thm}
	\begin{proof}
		Let $P$ stand for either $Q$ or $L$, and assume that $A^r=f(P)$.  If $r$ is odd or $G$ is not bipartite, then $G$ must be regular by Lemma~\ref{L-odd bip}, so we will assume that $G$ is bipartite and $r=2k$ for some $k$.  In this case we have $(A^2)^k=f(P)$, so by Lemma~\ref{L-same} any eigenvector of $P$ will also be an eigenvector of $A^2$.  If $P=Q$ and if $G$ is bipartite, then it is easy to see that  $\mathbf{1}'$ will be an eigenvector of $P$, and hence of $A^2$.  If $P=L$, then $\mathbf{1}$ is an eigenvector of $P$ and hence of $A^2$.  In either case we conclude that $G$ is biregular by Lemma~\ref{L-1}.
	\end{proof}
	
	\section{Relations Involving $\NL$}\label{S-NL}
	From the definition of $\mathcal{L}$ it is immediate that if $G$ is $d$-regular or $(d_1,d_2)$-biregular then $\mathcal{L}=I-\frac{1}{d}A$ or $\mathcal{L}=I-\frac{1}{\sqrt{d_1d_2}}A$, so when $G$ is regular or biregular it is possible to have $A^r=f(\NL)$  and $\NL^r=f(A)$.  
	
	We note the following (see \cite{chungButler}).  If $G$ is connected then $\Eig{\NL}{0}$ has dimension 1 and is spanned by $D^{1/2}\mathbf{1}$.  If $G$ is connected then $\Eig{L}{0}$ has dimension 1 and is spanned by $\mathbf{1}$.
	
	\begin{lem}\label{L-NL}
		If $D^{1/2}\mathbf{1}$ is an eigenvector of $A$, then $G$ is regular or biregular.
	\end{lem}
	\begin{proof}
		$D^{1/2}\mathbf{1}$ being an eigenvector of $A$ is equivalent to the statement that there exists a $\lambda$ such that $\sum_{uv\in E(G)}\sqrt{d_u}=\lambda \sqrt{d_v}$ for all $v$, or equivalently that for all $v$ $\frac{1}{\sqrt{d_v}}\sum_{uv\in E(G)}\sqrt{d_u}$ is the same value, $\lambda$.  Assume that this condition holds and let $v$ be a vertex of minimal degree $d$ and $v'$ a vertex of maximum degree $D$.  Then
		
		\[
		\lambda=\frac{1}{\sqrt{d}}\sum_{uv\in E}\sqrt{d_u}\le \sqrt{d D}\le \frac{1}{\sqrt{D}}\sum_{uv\in E}\sqrt{d_u}=\lambda,
		\]
		since the first sum has $d$ terms that are at most $\sqrt{D}$ and the second has $D$ terms that are at least $\sqrt{d}$.  We conclude that the inequalities are equalities, and hence that all vertices have degree $d$ or $D$, and that every neighbor of a vertex with degree $d$ has degree $D$ and vice versa.  If $d=D$ we conclude that $G$ is regular.  If $d\ne D$ we can partition vertices into those with degree $d$ and those with degree $D$, and this shows that $G$ is bipartite and hence biregular.
	\end{proof}
	\begin{thm}\label{T-NLA}
		If $G$ is connected and $\NL^r=f(A)$, then $G$ is regular or biregular.
	\end{thm}
	\begin{proof}
		$\Eig{\NL}{0}=\Eig{\NL^r}{0},\ \dim \Eig{\NL}{0}=1$ and $D^{1/2}\mathbf{1}\in \Eig{\NL}{0}$.  Thus if $\NL^r=f(A)$, then $D^{1/2}\mathbf{1}$ will be an eigenvector of $A$ by Lemma~\ref{L-onedim}, and this implies that $G$ is either regular or biregular by Lemma~\ref{L-NL}.
	\end{proof}
	
	\begin{lem}\label{L-NLQ}
		If $D^{1/2}\mathbf{1}$ is an eigenvector of $Q$, then $G$ is regular.
	\end{lem}
	\begin{proof}
		$D^{1/2}\mathbf{1}$ being an eigenvector of $A$ is equivalent to the statement that there exists a $\lambda$ such that $d_v\sqrt{d_v}+\sum_{uv\in E}\sqrt{d_u}=\lambda \sqrt{d_v}$ for all $v$, or equivalently that $d_v+\frac{1}{\sqrt{d_v}}\sum_{uv\in E}\sqrt{d_u}$ is the same for all $v$.  Assume that this condition holds and let $v$ be a vertex of minimal degree $d$ and $v'$ a vertex of maximum degree $D$.  Then
		
		\[
		\lambda=d+\frac{1}{\sqrt{d}}\sum_{uv\in E}\sqrt{d_u}\le d+ \sqrt{d D}\le D+\sqrt{d D}\le D+\frac{1}{\sqrt{D}}\sum_{uv\in E}\sqrt{d_u}=\lambda,
		\]
		since the first sum has $d$ terms that each have value at most $\sqrt{D}$ and the second has $D$ terms that each have value at least $\sqrt{d}$.  Thus every inequality must be an equality, and in particular this implies that $d=D$, so $G$ is regular.
	\end{proof}
	\begin{lem}\label{L-NLL}
		If $\mathbf{1}$ is an eigenvector of $\NL$, then $G$ is regular.
	\end{lem}
	\begin{proof}
		We have that $(\NL\mathbf{1})_v=1-\frac{1}{\sqrt{d_v}}\sum_{uv\in E} \frac{1}{\sqrt{d_u}}$, and that $\mathbf{1}$ is an eigenvector of $\NL$ only if this value is equal to the same value $\lambda$ for all $v$.  Assume this is true and let $v$ be a vertex of maximum degree $D$.  We then have that \[\lambda=1-\frac{1}{\sqrt{D}}\sum_{uv\in E} \frac{1}{\sqrt{d_u}}\le 1-\frac{1}{\sqrt{D}}\frac{D}{\sqrt{D}}=0,\] since the sum is minimized when each of the terms is equal to $1/\sqrt{D}$.  But $\lambda \ge 0$ (because the spectrum of $\NL$ is nonnegative), so this inequality must be an equality.  This implies that every vertex of maximum degree is adjacent only to vertices of maximum degree, and since $G$ is connected, we conclude that $G$ is regular of degree $D$.
	\end{proof}

	\begin{thm}
		If $G$ is connected and $\NL^r=f(Q)$ or $Q^r=f(\NL)$, then $G$ is regular.
	\end{thm}
	\begin{proof}
		Either case implies that $D^{1/2}\mathbf{1}$ is an eigenvector of $Q$ by Lemma~\ref{L-bothNonneg}, and this implies that $G$ is regular by Lemma~\ref{L-NLQ}.
	\end{proof}
	\begin{thm}
		If $G$ is connected and $\NL^r=f(L)$ or $L^r=f(\NL)$, then $G$ is regular.
	\end{thm}
	\begin{proof}
		Either case implies that $\mathbf{1}$ is an eigenvector of $\NL$ by Lemma~\ref{L-bothNonneg}, and this implies that $G$ is regular by Lemma~\ref{L-NLL}.
	\end{proof}
	
	Of relations involving the four matrices $A,\ Q,\ L,$ and $\NL$, the only remaining case is $A^r=f(\NL)$.  Unfortunately, we do not have a complete characterization for this case, though experimental data suggests the following conjecture.
	\begin{conj}\label{con-full}
		If $G$ is connected and $A^r=f(\NL)$ for some polynomial $f$ and $r>0$, then $G$ is regular or biregular.
	\end{conj}
	
	We present some partial results related to this conjecture.
	
	\begin{prop}\label{P-odd}
		If $A^r=f(\NL)$ with $r$ odd, then $G$ is regular or biregular.
	\end{prop}
	
	\begin{proof}
		If $r$ is odd then $\Eig{A^r}{\mu^r}=\Eig{A}{\mu}$ for all $\mu$.  If $A^r=f(\NL)$, then $D^{1/2}\mathbf{1}$ is an eigenvector of $f(\NL)$, and hence of $A^r$, and hence of $A$, implying that $G$ is regular or biregular by Lemma~\ref{L-NL}.
	\end{proof}
	We note the following conjecture, which again experimental data suggests is true.
	\begin{conj}\label{con-square}
		If $G$ is connected and $D^{1/2}\mathbf{1}$ is an eigenvector of $A^2$, then $G$ is regular or biregular.
	\end{conj}
	\begin{prop}
		If Conjecture~\ref{con-square} is true, then Conjecture~\ref{con-full} is true.
	\end{prop}
	\begin{proof}
		The case of Conjecture~\ref{con-full} when $r$ is odd is proved in Proposition~\ref{P-odd}.  If $r$ is even then we have $(A^2)^k=f(\NL)$, so $D^{1/2}\mathbf{1}$ will be an eigenvector of $A^2$ by Lemma~\ref{L-same}. Conjecture~\ref{con-square} being true then implies that $G$ is regular or biregular as desired.
	\end{proof}

	\section{General Polynomial Relations}\label{S-gen}
	A more general question one can ask is about the existence of nontrivial polynomials $f$ and $g$ such that  $f(X)=g(Y)$ for $X,\ Y$ matrices of a graph $G$.  By nontrivial we mean that $f$ and $g$ are not of the form $f=up+c,\ g=vq+c$ where $p,\ q$ are the minimal polynomials of $X$ and $Y$ respectively, $c$ is a constant, and $u,\ v$ are arbitrary polynomials.  When this occurs we have the following correspondence between eigenvalues of $X$ and eigenvalues of $Y$.
	
	\begin{prop}\label{P-general}
		Let $X$ and $Y$ be diagonalizable matrices with $f(X)=g(Y)$ for polynomials $f$ and $g$.  If $\lambda_1,\ldots,\lambda_n$ are the eigenvalues of $X$ and $\mu_1,\ldots,\mu_n$ are the eigenvalues of $Y$, then $\{f(\lambda_1),\ldots,f(\lambda_n)\}=\{g(\mu_1),\ldots,g(\mu_n)\}$.
	\end{prop}
	Note that this result holds even if $f$ and $g$ are trivial, but the conclusion isn't particularly interesting.
	\begin{proof}
		Let $Z=f(X)=g(Y)$ and let $V'=\{v_1,\ldots,v_n\}$ be a basis of eigenvectors of $X$ with $Xv_i=\lambda_iv_i$.  Then $Zv_i=f(X)v_i=f(\lambda_i)v_i$, so $Z$ will have eigenvalues $\{f(\lambda_1),\ldots,f(\lambda_n)\}$.  A symmetric argument shows that  $Z$ will have eigenvalues $\{g(\mu_1),\ldots,g(\mu_n)\}$, so these sets must be equal.
	\end{proof}
	
	For example, if $P_4$ denotes the path on 4 vertices then one can compute that
	\[
	A(A^2-2I)=Q^3-5Q^2+6Q-I.
	\]
	One can also compute that the eigenvalues of $A$ are $\frac{1+\sqrt{5}}{2},\ \frac{1-\sqrt{5}}{2},\ \frac{-1+\sqrt{5}}{2},\ \frac{-1-\sqrt{5}}{2}$, and that the eigenvalues of $Q$ are $2+\sqrt{2},\ 2-\sqrt{2},\ 2,\ 0$.  If $f(x)=x(x^2-1)$ and $g(x)=x^3-5x^2+6x-1$, then
	\begin{align*}
	f(\frac{1+\sqrt{5}}{2})=f(\frac{1-\sqrt{5}}{2})&=g(2+\sqrt{2})=g(2-\sqrt{2})\\ 
	f(\frac{-1+\sqrt{5}}{2})=f(\frac{-1-\sqrt{5}}{2})&=g(2)=g(0),
	\end{align*}
	which agrees with Proposition~\ref{P-general}.

	On the other hand, if $G$ denotes the graph which has the following adjacency matrix, then one can prove that there exists no nontrivial relation $f(A)=g(Q)$.
	
	\[
	A_{G}=\begin{bmatrix}
	0 & 1 & 1 & 1 & 1\\ 
	1 & 0 & 1 & 1 & 1\\ 
	1 & 1 & 0 & 1 & 0\\ 
	1 & 1 & 1 & 0 & 0\\ 
	1 & 1 & 0 & 0 & 0
	\end{bmatrix}.
	\]

	The idea of the proof is as follows.   One observes that  the minimal polynomial of $A$ has degree 4, which implies that every power of $A$ can be expressed as a polynomial of $A$ that has degree at most 3.  Thus if a nontrivial polynomial $f$ exists such that $f(A)=g(Q)$, it can be chosen to be of degree 3 or smaller.  The minimal polynomial of $Q$ is also of degree 4, so we again conclude that if $g$ exists it can be chosen to have degree at most 3.  In total, if $f,g$ exist then one can express them as a linear combination of matrices  from the set $\{I,A,A^2,A^3,Q,Q^2,Q^3\}$.  However, one can verify that this collection of matrices (thought of as $5^2$-dimensional vectors) are linearly independent, so there exist no nontrivial polynomials such that $f(A)=g(Q)$.
	
	There does not seem to be an obvious characterization of graphs that satisfy $f(X)=g(Y)$, nor does there seem to be a characterization of what these polynomials $f$ and $g$ look like when this occurs, but we have not investigated this question very thoroughly.  It also does not appear that one can refine Proposition~\ref{P-general} in such a way that, given the eigenvalues of $X$ and the relation $f(X)=g(Y)$, one can compute the eigenvalues of $Y$ in general, but there may exist special classes of relationships like $X^r=f(Y)$ for which this refinement is possible.
	
	One direction for future study would be to answer questions of the following type: let $P$ be a property that a graph can have (such as being $(d_1,d_2)$-biregular or being isomorphic to $P_n$) and two matrices of graphs $X$ and $Y$.  Can one give an explicit (nontrivial) relation $f(X)=g(Y)$ for all graphs satisfying $P$?  If so, can one use this explicit relation to directly relate the eigenvalues of $X$ and $Y$ for graphs satisfying  $P$?  For example, we have the theorem that if $G$ is $(d_1,d_2)$-biregular, then $A^2=(Q-d_1I)(Q-d_2I)$.  An example of another problem of this type is as follows:
	\begin{quest}
		Are there (non-trivial) functions $f_n,\ g_n$ such that $f_n(A)=g_n(Q)$ when $G=P_n$ for all $n$?  If so, can one give an explicit (nice) construction of such functions?
	\end{quest}
	
	Another direction to explore would be to generalize results like Theorem~\ref{T-two}, stating that the only graphs satisfying $A^r=f(Q)$ are those that are regular or biregular.  One could instead ask the following question: given matrices of graphs $X$ and $Y$ and a family of ordered pairs of polynomials $\mathcal{F}=\{(f,g)\}$, does there exist a (nice) property $P$ such that the only graphs satisfying $f(X)=g(Y)$ for some $(f,g)\in \mathcal{F}$ are those satisfying $P$?  For example, we have the following result.
	\begin{prop}\label{P-deg2}
		If $f(A)=g(L)$ where $f$ is a polynomial of degree at most 2 with nonnegative coefficients and $g$ is an arbitrary polynomial, then $G$ is regular or biregular.  Moreover, if $f$ can't be chosen to be $f(x)=x^2$, then $G$ is regular.
	\end{prop}
	\begin{proof}
		If these polynomials exist, choose them such  that  $f$ is monic and has no constant term.  Let $c$ denote the constant term of $g$.  If $f(x)=x$ then $A\mathbf{1}=g(L)\mathbf{1}=c\mathbf{1}$ (since $L\mathbf{1}=0$), and this implies that  $G$ is $c$-regular.  If $f(x)=x^2+ax$, then $A^2\mathbf{1}+aA\mathbf{1}=f(A)\mathbf{1}=g(L)\mathbf{1}=c\mathbf{1}$.  We conclude that $a d_v+\sum_{uv\in E(G)}d_u=c$ for all $v$ by using the same logic as in Lemma~\ref{L-1}.  If $v$ is a vertex of minimum degree $d$ and $v'$ is a vertex of maximum degree $D$ we have (noting that $a\ge 0$)
		\[
		c=ad+\sum_{uv\in E(G)}d_u\le ad+dD\le aD+dD\le aD+\sum_{uv'\in E(G)}d_u=c,
		\]
		so we conclude that all inequalities are equalities.  If $a\ne 0$ this implies  that  $d=D$, making $G$ regular.  If $a=0$ and $d\ne D$, then one can partition the vertices of $G$ into those with degree $d$ and those with degree $D$, making $G$ biregular.
	\end{proof}
	We note that the assumption that $f$ have nonnegative coefficients can not be relaxed.  Indeed, let $G'$ be defined by the adjacency matrix 
	
	\begin{equation}\label{eq}
	A_{G'}=\begin{bmatrix}
	0 & 1 & 1 & 1 & 1\\ 
	1 & 0 & 0 & 0 & 1\\
	1 & 0 & 0 & 1 & 0\\
	1 & 0 & 1 & 0 & 0\\ 
	1 & 1 & 0 & 0 & 0
	\end{bmatrix}.
	\end{equation}
	One can check that in this case $3A^2-3A=-L^3+9L^2-20L+12I$.

	We have analogous results for the signless Laplacian.
	
	\begin{prop}\label{P-Qdeg2}
		If $f(Q)=g(L)$ where $f$ is a polynomial of degree at most 2 with nonnegative coefficients and $g$ is an arbitrary polynomial, then $G$ is regular.
	\end{prop}
	\begin{proof}
		If these polynomials exist, choose them such that  $f$ is monic and has no constant term and let $c$ denote the constant term of $g$.  To proceed as in Proposition~\ref{P-deg2}, we will need to understand how $\mathbf{1}$ interacts with $Q$ and $Q^2$.  It is clear that $(Q\mathbf{1})_v=2d_v$.  For $Q^2$ we have
		\[
		Q^2=(A+D)^2=A^2+AD+DA+D^2
		\]
		We know that $(A^2\mathbf{1})_v=\sum_{uv\in E(G)}d_u$, and it isn't difficult to see that \[(AD\mathbf{1})_v=\sum_{uv\in E(G)}d_u,\ (DA\mathbf{1})_v=d_v^2,\ (D^2\mathbf{1})_v=d_v^2.\]  Thus in total we have $(Q^2\mathbf{1})_v=2d_v^2+2\sum_{uv\in E(G)}d_u$.
		
		If $f(x)=x$, then $Q\mathbf{1}=f(L)\mathbf{1}=c\mathbf{1}$, and this implies that $G$ is $c/2$-regular.  If $f(x)=x^2+ax$ then we conclude that $Q^2\mathbf{1}+aQ\mathbf{1}=c\mathbf{1}$.  By comparing the $v$th coordinates of both sides, we see that $d_v^2+ad_v+\sum_{uv\in E(G)}d_u=c/2$ and this holds for all $v$.  If $v$ denotes a vertex with minimum degree $d$ and $v'$ a vertex with maximum degree $D$ then
		\begin{align*}
		c/2=d^2+ad+&\sum_{uv\in E(G)}d_u\le d^2+ad+dD \le\\  D^2+aD+dD&\le
		D^2+aD+\sum_{uv'\in E(G)}d_u=c/2,
		\end{align*}
		so the inequalities must be equalities and we conclude that $d=D$, making $G$ regular.
	\end{proof}
	Again the condition that $f$ have nonnegative coefficients can not be weakened.  Indeed, if $G$ is a $(d_1,d_2)$-biregular graph then $(Q-d_1I)(Q-d_2I)=(L-d_1I)(L-d_2I)$.  While  biregular graphs are the most obvious counterexample, they are not the only ones.  For example, if we consider $G'$ as defined in \eqref{eq}, then one can show that $3Q^2-21Q=-2L^3+15L^2-25L-24I$.
	
	One can also ask whether polynomials $f$ and $g$ exist such that $f(X)=g(Y)$ when $X$ is a matrix associated to a graph $G$ and $Y$ is not.  One such example is $Y=J$, the $n\times n$ matrix whose entries are all 1.  Note that $J$ has rank 1, so the only non-trivial polynomials of $J$ are of the form $cI+dJ$ with $d\ne 0$.  Thus if $f$ and $g$ exist such that $f(X)=g(J)$, one can always choose $g(J)=J$.
	
	\begin{lem}\label{L-Jreg}
		If $f(A)=J$ for some polynomial $f(x)$, then $G$ is regular and connected.
	\end{lem}
	Note that all the matrices that we considered earlier had the property that $X_G=X_{G_1}\oplus X_{G_2}$ whenever $G$ was the disjoint union of the graphs $G_1$ and $G_2$.  This meant that the relation $f(X_G)=g(Y_G)$ held iff $f(X_{G'})=g(Y_{G'})$ held for any connected component $G'$ of $G$.  This is not the case when considering $J$, so we emphasize here the fact that $G$ must be connected.
	
	\begin{proof}
		Assume that such an $f$ exists.  If vertex $i$ and vertex $j$ belong to different components of $G$, then for all $r$, $A^r_{ij}=0$, which implies that there exists no polynomial such that $f(A)=J$.  It follows that $G$ must be connected.
		
		If $\pf$ is the positive eigenvector of $A$ guaranteed by the Perron-Frobenius theorem, then
		\[
		J\pf=f(A)\pf=f(\Lambda)\pf,
		\]
		so $\pf$ is an eigenvector of $J$, but the only positive eigenvectors of $J$ are scaler multiples of $\mathbf{1}$, so $\pf=c\mathbf{1}$ for some $c$, which means $G$ must be regular.
	\end{proof}
	Let $m_A(x)$ denote the minimal polynomial of $A$.  If $G$ is a $k$-regular graph, let $m'_A(x)=m_A(x)/(x-k)$.  Note that $m'_A(x)=\prod_i(x-\lambda_i)$, where the $\lambda_i$ range over all distinct eigenvalues of $A$ that are not equal to $k$.
	\begin{lem}\label{L-divide}
		If $f(A)=J$, then $f(x)\mid m'_A(x)$.
	\end{lem}
	\begin{proof}
		By Lemma~\ref{L-Jreg} we can assume that $G$ is $k$-regular for some $k$.  Let $v$ be an eigenvector of $A$ with eigenvalue $\lambda\ne k$.  We have
		\[
		Jv=f(A)v=f(\lambda)v,
		\]
		so $v$ is an eigenvector of $J$ with eigenvalue $f(\lambda)$.  But  $v\ne c\mathbf{1}$ by assumption of $\lambda\ne k$, so it must be that $v$ is a null-vector of $J$ and $f(\lambda)=0$.  As every distinct eigenvalue  of $A$ not equal to $k$ is a root of $f$, we conclude that $f(x)\mid m'_A(x)$.
	\end{proof}
	\begin{thm}
		A polynomial $f(x)$ exists such that $f(A)=J$ iff $G$ is connected and regular.  Moreover, if $f$ is chosen to have minimum degree, then $f(x)=cm'_A(x)$ for some $c\ne 0$.
	\end{thm}
	\begin{proof}
		Lemma~\ref{L-Jreg} gives the forward direction, so assume $G$ is connected and $k$-regular.  This implies that the null-space of $A-kI$ has dimension 1 and is spanned by $\mathbf{1}$.  We also have \[(A-kI)m'_A(A)=m_A(A)=0.\]  These two facts imply that every column of $m'_A(A)$ is a scaler multiple of $\mathbf{1}$.  As $m'_A(A)$ is a symmetric matrix, we must have $m'_A(A)=c\mathbf{1}\mathbf{1}^T=cJ$ for some $c\ne  0$, so $f(x)=\frac{1}{c}m'_A(x)$ gives the desired polynomial.
		
		Finally, assume $f(A)=J$ where $f(x)$ is chosen to have minimum degree.  From the above proof we know that $\deg(f)$ can be at most $\deg(m'_A)$, but by Lemma~\ref{L-divide} we must have $\deg(f)\ge \deg(m'_A)$.  We conclude that $\deg(f)=\deg(m'_A)$ and that $f(x)=cm'_A(x)$ for some $c\ne 0$.
	\end{proof}  
	The above statement implies that if $G$ is a connected, $k$-regular graph such that $A$ has $r+1$ distinct eigenvalues, then there exists $c,d\ne 0$ such that $cm'_A(x)$ is a monic polynomial of degree $r$, and $cm'_A(A)=dJ$. The $r=2$ case corresponds to connected strongly regular graphs, which are usually defined combinatorially as is done so in \cite{AGT}, for example.  Given any connected strongly regular graph, one can derive an equation of the form $cm'_A(A)=dJ$ by having the coefficients of the polynomial be defined in terms of combinatorial parameters of the graph.  It would be interesting to know if this process could be reversed in general.  That is, can one always interpret the coefficients of the equation $cm'_A(A)=dJ$ in terms of certain parameters of the underlying graph, and can these parameters be used to give a combinatorial description of connected regular graphs with precisely $r+1$ eigenvalues?
	
	\section{Acknowledgments}
	The author would like to thank Richard Stanley for suggesting this research topic, as well as his assistance with the general structure of the paper.
	
	\bibliographystyle{plain}
	\bibliography{Matrices}
\end{document}